\documentclass[11pt]{amsart}
\usepackage{amssymb,latexsym,amsmath}

\theoremstyle{plain}
\newtheorem{theorem}{Theorem}[section]
\newtheorem{proposition}[theorem]{Proposition}
\newtheorem{lemma}[theorem]{Lemma}
\newtheorem{corollary}[theorem]{Corollary}

\theoremstyle{remark}
\newtheorem{remark}[theorem]{Remark}

\theoremstyle{definition}



\newcommand{\R}{\mathbb{R}}

 \newcommand{\Pu}{\mathbb{P}}

\newcommand{\Vol}{\textup{Vol}}

\begin{document}

\title[On the isotropy constant of projections of polytopes]
{On the isotropy constant of projections of polytopes}

\author[D.\,Alonso]{David Alonso-Guti\'errez}
\address{Dpto. de Matem\'aticas, Facultad de Ciencias, Universidad de
Zaragoza, 50009 Zaragoza, Spain}\email[(David
Alonso)]{daalonso@unizar.es}
\thanks{Partially supported by Marie Curie RTN CT-2004-511953. The three first named authors partially supported by MCYT
Grant(Spain) MTM2007-61446 and DGA E-64. The fourth named author partially supported by National Science Foundation FRG grant DMS-0652722 (U.S.A.).}

\author[J.\,Bastero]{Jes\'us Bastero}\email[(Jes\'us
Bastero)]{bastero@unizar.es}

\author[J.\,Bernu\'es]{Julio Bernu\'es}
\email[(Julio Bernu\'es)]{bernues@unizar.es}

\author[P. Wolff]{Pawe\l{} Wolff}\email[(Pawe\l{} Wolff)]{pawel.wolff@case.edu}
\address{Department of Mathematics, Case Western Reserve University, Cleveland, OH 44106-7058, U.S.A.}



\begin{abstract} The isotropy constant of any $d$-dimensional polytope with $n$ vertices is bounded by $C \sqrt{n/d}$ where $C>0$ is a numerical constant.

\end{abstract}

\maketitle

\section {Introduction}

The boundedness of the isotropy constant (see definition below)
is a major conjecture in Asymptotic Geometric Analysis.
The answer is known to be positive for many families of convex bodies, see for instance \cite{MP}
or \cite{KK} and the references therein. In this paper we
focus our attention on the isotropy constant of polytopes or, equivalently,
of projections of the unit ball of $\ell_1^n$ space (in the symmetric case) and of the regular $n$-dimensional
simplex $S_n$ (in the non-symmetric case).\medskip

M.~Junge~\cite{J1} proved that the isotropy constant of all
orthogonal projections of $B^n_p$,  the unit ball of the
$\ell_p^n$ space, $1< p\le\infty$, is bounded by $Cp\prime$
an estimate improved to $C\sqrt{p\prime}$ in \cite{KM} ($p'$ is the conjugate exponent of $p$ and $C$ a numerical
constant). Later \cite[Theorem 4]{J2} M. Junge showed that the isotropy constant of any
symmetric
polytope with $2n$ vertices is bounded by $C\log n$, see also
\cite{milman-dual-mixed-volumes-and-slicing-problem}.\medskip

In a recent paper, \cite{KK}, B. Klartag and G. Kozma show the boundedness of the isotropy constant of random Gaussian polytopes. The integral over a polytope, which defines its isotropy constant, is computed by passing to an integral over its surface (faces). A consequence of their results is that ``most'' (see precise meaning below) $d$-dimensional projections of $B^n_1$ as well as of $S_n$ have bounded isotropy constant. When reading this statement one should have in mind
the well-known fact 
that \emph{every} symmetric convex body in $\R^d$ is ``almost" a projection of a $B_1^n$-ball, with possibly large $n$.  In the same spirit, positive answers for other random $d$-dimensional polytopes with $n\ (\ge Cd)$ vertices were given in \cite{A} and \cite{DGG}.\medskip

Our main theorem (Corollary \ref{isotropy-constant-polytopes}) states that for any $d$-dimensional polytope $K$ with $n$ vertices its isotropy constant $L_K$ verifies $$L_K\le C \sqrt{\frac{n}{d}}$$ where $C>0$ is a numerical constant.\medskip

We now pass to describe the contents of the paper. The second section introduces the geometric tool
(Proposition~\ref{betke:main-proposition}) necessary
to deal with integration on $d$-dimensional
projections of polytopes (Corollary~\ref{betke:main-corollary}).
Some time ago, one of the authors learned about this tool from Prof. Franck Barthe. The ideas originate from a paper
by U.~Betke~\cite{betke}, where a general result was
presented, namely a related formula for mixed volumes of two
polytopes. However, for the sake of completeness, we provide the
proof of the particular result we need. It also seems that the content of the proof is more geometric. \medskip

In the third Section we use these tools to prove our aforementioned main result (Corollary~\ref{isotropy-constant-polytopes}) by easily reducing it to the cases $K=P_E B^n_1$ or $K=P_E S_n$ (Theorem~\ref{isotropic-constant-of-proj-of-b1n-ball}) where $E\subset\R^n$ is any $d$-dimensional subspace and $P_E$ is the orthogonal projection onto $E$. Also in this Section we give a proof of the observation that for ``most'' subspaces, that is, for a subset $A$ of the Grassmann space $G_{n,d}$ of Haar probability measure $\ge 1-c_1e^{-c_2\max\{\log n,d\}}$, one has $L_{P_E B_1^n}<C $ and $ L_{P_E S_n}<C$ for every $E\in A$ with numerical constants $C,c_1,c_2$ (Proposition~\ref{prop-3-3}).\medskip

The next Section studies the isotropy constant of projections of
random polytopes with vertices on the sphere $S^{n-1}$. Using the techniques from Section 2 and \cite{A} we show that, with high probability, the isotropy constants of \emph{all} $d$-dimensional projections of random polytopes are bounded by $C\sqrt{\frac{n}{d}}$ (Proposition~\ref{prop-4-1}).\medskip

In the last Section we show that for every isotropic convex body, the isotropy constants of its hyperplane projections are comparable to the isotropy constant of the body itself (Corollary~\ref{AlonsoMP}). Recall that the analogous result for hyperplane \textit{sections} was already proved in \cite{MP}. The proof uses Steiner symmetrization in a similar way as it appears in \cite{BKM}, with better numerical constants.  In particular, we have $L_{P_H B^n_p}\le C$ for any hyperplane $H$ and $1<p<\infty$ improving Junge's estimate \cite{J1} for the case of hyperplanes. In \cite{ABBW} a different proof of this fact is given with the hope it might be extended to lower dimensional projections. \medskip

We recall that a convex body $K\subset\R^n$ is \emph{isotropic} if
it has volume $\Vol_n(K)=1$, the barycenter of $K$ is at the origin and its inertia matrix is a multiple
of the identity. Equivalently, there exists a constant $L_K>0$
called isotropy constant of $K$ such that $L_K^2=\int_K\langle
x,\theta\rangle^2\, dx, \forall \theta\in S^{n-1}$.

It is well known~\cite{MP}, that every convex body $K\subset\R^n$
has an affine transformation $K_1$ isotropic, so we can write
$L_K:=L_{K_1}$. This is well defined and moreover,

\begin{equation}\label{Milman-Pajor}
nL_K^2=\inf\left\{\frac{1}{\Vol_n(K)^{\frac2{n}+1}}\int_{a+TK}|x|^2\,
dx;\, a\in\R^n, T\in SL(n)\right\}\end{equation}

For a convex body $K\subset\R^n$,  $R(K)=\max\{|x|\, : x\in K\}$ and $r(K)=\min\{|x|\, : x\in
K\}$ are the circumradius and the inradius of $K$ respectively.

We will think of $S_n$ as an $n$-dimensional regular simplex in $\R^n$ with center of mass at the origin.
We will write $\Delta_n=\textup{conv}\{e_1,\dots ,e_{n+1}\}$ for the natural position of an $n$-dimensional regular simplex in $\R^{n+1}$.

The Lebesgue measure on an affine subspace $E$ will be denoted by $\lambda_E$.
For a measurable set $A \subseteq E$, if $d$ is a dimension of $E$, $\Vol_d(A)$ will stand for $\lambda_E(A)$.

The notation $a\sim b$ means $a\cdot c_1\le b\le a\cdot c_2$ for
some numerical constants $c_1, c_2>0$.

\section{Projections of polytopes}

Throughout this section, $K \subseteq \R^n$ is a polytope
(non-empty but possibly of empty interior), $E
\subseteq \R^n$ is a linear subspace of dimension $d$ ($1 \le d \le
n-1$) and $P_E$ is the orthogonal projection onto $E$.
\medskip

Let us fix some notation and recall necessary definitions (we follow the book by Schneider~\cite[Ch. 1, 2]{schneider}).
For a subset $A \subseteq \mathbb R^n$, $\textup{aff} \, A$ denotes the
minimal affine subspace which contains $A$. The dimension of a convex set $A$ is $\textup{dim} \, \textup{aff} \, A$.
When writing $\textup{relint} \, A$ we mean the relative interior of $A$ w.r.t. the topology of $\textup{aff} \, A$.
If $G \subseteq \mathbb R^n$ is an affine subspace then $G_0$ denotes the linear subspace
parallel to $G$. A convex subset $F \subseteq K$ of a polytope $K$ is called \emph{a face} if for any $x, y \in K$, $(x+y)/2 \in F$ implies
$x, y \in F$ (see also~\cite[Sec. 1.4, pp. 18]{schneider}). The set of $j$-dimensional
faces ($j$-faces, in short) of $K$ will be denoted as $\mathcal{F}_j(K)$ ($j = 0, 1, \ldots, n$), and $\mathcal{F}(K) =
 \bigcup_{j=0}^n \mathcal{F}_j(K) \cup \{ \emptyset \}$ is the set of all faces of $K$ ($\emptyset$ is also a face).
$K$ can be decomposed into a disjoint union of $\{ \textup{relint} \, F \, ; \; F \in \mathcal{F}(K) \}$
(see {\cite[Thm. 2.1.2]{schneider}}). For that reason for any $x \in K$ the unique face $F \in \mathcal{F}(K)$ such that $x \in \textup{relint}\, F$
will be denoted by $F(K, x)$.

For $x \in K$, \emph{a normal cone} of $K$ at $x$ is
\[
  N(K, x) = \{ u \in \mathbb R^n \, ; \; \forall_{z \in K} \, \langle z - x, u \rangle \le 0 \}.
\]
$N(K, x)$ is a closed convex cone. We shall also consider another closed convex cone, namely
\[
  S(K, x) = \bigcup_{\lambda > 0} \lambda (K-x).
\]
(In general, i.e. when $K$ is a convex body, $S(K, x)$ does not have to be closed.)
By~\cite[(2.2.1)]{schneider},
\begin{equation}
  \label{betke:crucial-duality-of-cones-1}
  N(K, x)^\ast = S(K, x),
\end{equation}
where the polarity used here is the polarity of convex cones, namely, if $C \subseteq \mathbb R^n$ is a convex cone,
\[
  C^\ast = \{ y \in \mathbb R^n \, ; \; \forall_{x \in C} \, \langle x, y \rangle \le 0 \}
\]
(see also~\cite[Sec. 1.6, pp. 34]{schneider}).
We shall also need to consider normal cones taken w.r.t. an affine subspace.
If $G$ is an affine subspace of $\mathbb R^n$ and $L \subseteq G$ is a convex body,
then for $x \in L$ we define a normal cone for $L$ at $x$ taken w.r.t. $G$:
\[
  N_G(L, x) = \{ u \in G_0 \, ; \; \forall_{z \in L} \, \langle z - x, u \rangle \le 0 \}.
\]
Note that $N_G(L, x) \subseteq G_0$. The similar duality relation to (\ref{betke:crucial-duality-of-cones-1}) holds:
\begin{equation}
  \label{betke:crucial-duality-of-cones-2}
  N_G(L, x)^{\ast_{G_0}} = S(L, x),
\end{equation}
where the polarity is taken w.r.t. $G_0$.

For any face $\emptyset \neq F \in \mathcal{F}(K)$, define
\[
  N(K,F) := N(K, x), \quad \textup{where } x \in \textup{relint} \, F.
\]
This definition does not depend on the choice of $x$
(see~\cite[Sec. 2.2., pp. 72]{schneider}).  $N_G(L,F)$ is
analogously defined.\medskip

For a given polytope $K \subseteq \R^n$ and a linear subspace $E \subseteq \R^n$ of dimension $d$, let us fix any
$u \in E^\perp \setminus \{0\}$ which satisfy
\begin{equation}
  \label{betke:condition-on-u}
  u \notin \bigcup \{ P_{E^\perp} N(K, F) \, ; \; F \in \mathcal{F}(K) \setminus \{\emptyset\}, \textup{dim} \, P_{E^\perp} N(K, F) \le n-d-1 \}.
\end{equation}
Clearly, such $u$ exists, since (\ref{betke:condition-on-u})
excludes only a finite union of sets of  dimension $< n-d$ from
$E^\perp$ which is of dimension $n-d$.\medskip

Consider the following subsets of $\mathcal{F}(K)$:
\[ \begin{split}
  \tilde{\mathcal{F}}(K, E, u) &:= \{ F \in \mathcal{F}(K) \, ; \; u \in P_{E^\perp} N(K, F) \}, \\
  \tilde{\mathcal{F}}_d(K, E, u) &:= \tilde{\mathcal{F}}(K, E, u) \cap \mathcal{F}_d(K).
\end{split} \]

\begin{proposition}
  \label{betke:main-proposition}
  Let $K\subset \R^n$ be a polytope, $E$ a $d$-dimensional subspace, $u\in E^{\perp}$ verifying (\ref{betke:condition-on-u}) and $\tilde{\mathcal{F}}(K, E, u)$  as described above. Then
  \begin{enumerate}
    \item[(a)] $\{ P_E (\textup{relint} \, F) \, ; \; F \in \tilde{\mathcal{F}}(K, E, u) \}$ is a family of pair-wise disjoint sets,
    \item[(b)] $\bigcup \{P_E F \, ; \; F\in\tilde{\mathcal{F}}(K, E, u)\} = P_E K$.
  \end{enumerate}
  Moreover, $\tilde{\mathcal{F}}(K, E, u) \subseteq \bigcup_{0 \le j \le d} \mathcal{F}_j(K)$ and
  for each $F \in \tilde{\mathcal{F}}_d(K, E, u)$, $P_E |F \colon F \to P_E F$ is an affine isomorphism.
\end{proposition}
In the proof of the proposition we shall use several lemmas.
\begin{lemma}
  \label{betke:projection-of-normal-cone-lemma}
  Let $L$ be a polytope in $\R^n$ and $G \subseteq \R^n$ be an affine subspace. If $x \in L \cap G$ then
  \[
    P_{G_0} N(L, x) = N_G(L \cap G, x).
  \]
\end{lemma}
\begin{proof}
  By taking polars w.r.t. $G_0$ we see that the assertion is equivalent to
  \begin{equation}
    \label{betke:eq-for-normal-cones}
    N(L, x)^\ast \cap G_0 = N_G(L \cap G, x)^{\ast_{G_0}}
  \end{equation}
  (for the l.h.s. we used the fact that for a convex cone $C$, $(P_{G_0} C)^{\ast_{G_0}} = C^\ast \cap G_0$).
  Since $G_0 = G - x$ by~(\ref{betke:crucial-duality-of-cones-1}) we get
  \[ \begin{split}
    N(L, x)^\ast \cap G_0 &= S(L, x) \cap (G - x) = \bigcup_{\lambda > 0} \lambda (L-x) \cap (G - x) \\
                          &= S(L \cap G, x).
  \end{split} \]
  Applying~(\ref{betke:crucial-duality-of-cones-2}) we see that the r.h.s. of (\ref{betke:eq-for-normal-cones}) is also equal to
  $S(L \cap G, x)$.
\end{proof}
\begin{lemma}{\cite[Sec. 2.2]{schneider}}
  \label{betke:normal-cones-in-vertices-lemma}
  Let $L$ be a polytope contained in an affine subspace $G \subseteq \R^n$. Then
  \[
    \bigcup_{F \in \mathcal{F}_0(L)} N_G(L, F) = G_0.
  \]
\end{lemma}
\begin{lemma}
  \label{betke:intersection-of-normal-cones-lemma}
  With the hypothesis as in the previous lemma, for $x, y \in L$,
  \[
    N_G(L, x) \cap N_G(L, y) = N_G(L, (x+y)/2).
  \]
\end{lemma}
\begin{proof}
  The inclusion $\subseteq$ is immediate from the definition of a normal cone. For the converse inclusion take $u \in N_G(L, (x+y)/2)$.
  Then $\langle x - \frac{x+y}{2}, u \rangle \le 0$, $\langle y - \frac{x+y}{2}, u \rangle \le 0$, so
  $\langle x-y, u \rangle = 0$. Now, for all $z \in L$,
  \[
    \langle z-x, u \rangle = \langle z - \frac{x+y}{2}, u \rangle + \langle \frac{y-x}{2}, u \rangle \le 0,
  \]
  so $u \in N_G(L, x)$. Similarly $u \in N_G(L, y)$.
\end{proof}
\begin{lemma}{\cite[Sec. 2.4]{schneider}}
  \label{betke:dim-of-normal-cone-lemma}
  With the hypothesis as in Lemma~\ref{betke:normal-cones-in-vertices-lemma}, for $\emptyset \neq F \in \mathcal{F}(L)$,
  \[
    \textup{dim} \, N_G(L, F) = \textup{dim} \, G - \textup{dim} \, F
  \]
\end{lemma}
\begin{remark}
  Actually we shall use only the inequality $\textup{dim} \, N_G(L, F) \le \textup{dim} \, G - \textup{dim} \, F$ which simply
  follows from the fact $N_G(L, F) \subseteq \big((\textup{aff} \, F)_0\big)^{\perp_{G_0}}$.
\end{remark}
\begin{lemma}
  \label{betke:dim-of-faces-lemma}
  Let $K \subseteq \R^n$ be a polytope, $E \subseteq \R^n$ be a linear subspace of dimension $d$
  and $u \in E^\perp$ satisfies~(\ref{betke:condition-on-u}).
  Let $y \in K$, $x = P_E y \in E$, $K_x = K \cap (x + E^\perp)$
  ($K_x$ is a polytope in $x + E^\perp$).
  If one of the equivalent condition holds:
  \begin{enumerate}
    \item[(i)] $u \in P_{E^\perp} N(K, y)$,
    \item[(ii)] $u \in N_{x + E^\perp}(K_x, y)$,
  \end{enumerate}
  then $\{y\} \in \mathcal{F}_0(K_x)$ and $\textup{dim} \, F(K,y) \le d$.
\end{lemma}
\begin{proof}
  The conditions (i) and (ii) are equivalent by Lemma~\ref{betke:projection-of-normal-cone-lemma}.
  Consider $F = F(K, y)$ and $F' = F(K_x, y)$.
  By the condition~(\ref{betke:condition-on-u}) on $u$,
  \[
    \textup{dim} \, P_{E^\perp} N(K, F) \ge n-d,
  \]
  so $\textup{dim} \, N_{x+E^\perp}(K_x, F') \ge n-d$ and also $\textup{dim} \, N(K, F) \ge n-d$.
  Therefore Lemma~\ref{betke:dim-of-normal-cone-lemma} applied to $K_x$ and $F'$ implies $\textup{dim} \, F' = 0$, so $\{y\} = F' \in \mathcal{F}_0(K_x)$. Eventually, applying the same lemma to $K$ and $F$ yields $\textup{dim} \, F \le d$.
\end{proof}
\noindent\emph{Proof of Proposition~\ref{betke:main-proposition}.}
  \paragraph{(a)}
    Take $F_1, F_2 \in \tilde{\mathcal{F}}(K, E, u)$ such that for some $x \in E$,
    \[
      x \in P_E (\textup{relint} \, F_1) \cap P_E (\textup{relint} \, F_2)
    \]
    which means that for $i=1,2$ one can find $y_i \in x + E^\perp$ that $y_i \in \textup{relint} \, F_i$ and then
    \[
      u \in P_{E^\perp} N(K, F_i) = P_{E^\perp} N(K, y_i).
    \]
    Consider a convex polytope $K_x = K \cap (x + E^\perp)$. Lemma~\ref{betke:dim-of-faces-lemma} implies that
    $\{y_1\}, \{y_2\} \in \mathcal{F}_0(K_x)$ and also
    \[ \begin{split}
      u &\in N_{x+E^\perp}(K_x, y_1) \cap N_{x+E^\perp}(K_x, y_2) \\
        &=   N_{x+E^\perp}(K_x, (y_1+y_2)/2),
    \end{split} \]
    where the last equality is due to Lemma~\ref{betke:intersection-of-normal-cones-lemma}.
    But again, Lemma~\ref{betke:dim-of-faces-lemma} implies that also $\{(y_1+y_2)/2\} \in \mathcal{F}_0(K_x)$, hence $y_1 = y_2$
    (see definition of a face) and consequently, $F_1 = F(K, y_1) = F_2$.
  \paragraph{(b)}
    The inclusion ``$\subseteq$'' is obvious. For the inclusion ``$\supseteq$'' take arbitrary $x \in P_E K$.
    Put $K_x = K \cap (x+E^\perp)$. $K_x$ is a non-empty polytope in $x + E^\perp$. By Lemma~\ref{betke:normal-cones-in-vertices-lemma} one can
    find $y \in x+E^\perp$ such that $\{y\} \in \mathcal{F}_0(K_x)$ and $u \in N_{x+E^\perp}(K_x, y)$. Lemma~\ref{betke:dim-of-faces-lemma}
    (or just Lemma~\ref{betke:projection-of-normal-cone-lemma}) implies $u \in P_{E^\perp} N(K, F)$ where $F = F(K, y)$.
    Consequently, $F \in \tilde{\mathcal{F}}(K, E, u)$.

    By the definition of $\tilde{\mathcal{F}}(K, E, u)$ and Lemma~\ref{betke:dim-of-faces-lemma}, any face $F \in \tilde{\mathcal{F}}(K, E, u)$
    has dimension $\le d$.

    Finally we show that for $F \in \tilde{\mathcal{F}}_d(K, E, u)$, $P_E | F \colon F \to P_EF$ is an isomorphism.
    The condition~(\ref{betke:condition-on-u}) and Lemma~\ref{betke:dim-of-normal-cone-lemma} implies
    \[
      n-d \le \textup{dim} \, P_{E^\perp} N(K, F) \le \textup{dim} \, N(K, F) \le n - \textup{dim}\,F = n-d,
    \]
    so $\textup{dim} \, P_{E^\perp} N(K, F) = \textup{dim} \, N(K, F) = n-d$. This means
    $\big(\textup{span} \, N(K, F)\big) \cap E = \{0\}$ and $N(K, F)^\perp \cap E^\perp = \{0\}$. The definition of a normal cone
    yields $(\textup{aff} \, F)_0 \subseteq N(K, F)^\perp$, which finally gives $(\textup{aff} \, F)_0 \cap E^\perp = \{0\}$.
\null\hfill$\Box$\bigskip

The following corollary is an immediate consequence of
Proposition~\ref{betke:main-proposition}.
\begin{corollary}
  \label{betke:main-corollary}
  Let $K \subseteq \R^n$ be a convex polytope, $E \subseteq \R^n$ a $d$-dimensional linear subspace ($1 \le d \le n-1$).
  Then there exists a subset $\tilde{\mathcal{F}}$ of $\mathcal{F}_d(K)$ such that for any integrable function $f \colon E \to \R$,
  \begin{equation}
    \label{betke:main-equation}
    \int_{P_E K} f(x) \, \lambda_E(dx) =
      \sum_{F \in \tilde{\mathcal{F}}} \frac{\Vol_d(P_E F)}{\Vol_d(F)} \int_F f(P_E y) \, \lambda_{\textup{aff} F}(dy).
  \end{equation}
  In particular (for $f \equiv 1$),
  \[
    \Vol_d(P_E K) = \sum_{F \in \tilde{\mathcal{F}}} \Vol_d(P_E F).
  \]
\end{corollary}
\begin{proof}
  Choose any $u \in E^\perp$ satisfying (\ref{betke:condition-on-u}) for $K$ and $E$ and put $\tilde{\mathcal{F}} = \tilde{\mathcal{F}}_d(K, E, u)$.
  By Proposition~\ref{betke:main-proposition},
  \[ \begin{split}
    \int_{P_E K} f(x) \, \lambda_E(dx) &= \sum_{F \in \tilde{\mathcal{F}}} \int_{P_E F} f(x) \, \lambda_E(dx) \\
                                     &= \sum_{F \in \tilde{\mathcal{F}}}
                                       \frac{\Vol_d(P_E F)}{\Vol_d(F)} \int_{F} f(P_E y) \, \lambda_{\textup{aff} F}(dy).
  \end{split} \]
\end{proof}
For our purposes we shall use above corollary with $f(x) = |x|^2$.
In such case, the obvious inequality $|P_E y| \le |y|$ and the identity
$\displaystyle 1= \sum_{F \in \tilde{\mathcal{F}}}
\frac{\Vol_d(P_E F)}{\Vol_d(P_E K)}$  lead to the following
estimate: if $P_E K$ is a body of dimension $d$ (i.e. is
non-degenerated) then
\begin{equation}
  \label{betke:main-inequality}
     \frac{1}{\Vol_d(P_EK)} \int_{P_E K} |x|^2 \, \lambda_E(dx) \le
      \max_{F \in \tilde{\mathcal{F}}} \frac{1}{\Vol_d(F)} \int_F |y|^2 \, \lambda_{\textup{aff} F}(dy).
\end{equation}

\section{Projections of the $\ell^n_1$-ball and the regular simplex}

\bigskip
First of all, we are going to see that ``most'' projections of $B_1^n$ on $d$-dimensional subspaces
($d \le n$) have the isotropy constant bounded. It is well known that any symmetric convex polytope in $\R^d$ with $2n$
vertices is linearly equivalent to $P_E B_1^n$ for some $E\in G_{n,d}$.
Indeed, if $T \colon \R^n \to \R^d$ is a linear transformation
of full rank, then taking the $d$-dimensional subspace $E = (\textup{ker}\,T)^\perp \subseteq \R^n$, $T$ can be represented as $T_{|E} P_E$
where $T_{|E} \colon E \to \R^d$ is a linear isomorphism
being a restriction of $T$ to the subspace $E$. As an immediate consequence we obtain
the following
\begin{lemma}
  \label{symm-polytope-equivalence}
Let $K=\textup{conv}\{\pm v_1,\dots,\pm v_n\}\subseteq\R^d$ be a symmetric convex polytope with non-empty interior
and let $T \colon \R^n\to\R^d$ be the linear map such that $Te_i=v_i$.
Then for $E=(\textup{ker}\,T)^\perp \in G_{n,d}$, $P_E B_1^n$ and
$K$ are linearly equivalent.
\end{lemma}
One may also prove a similar lemma in the non-symmetric case. Recall that  $\Delta_n=\textup{conv}\{e_1,\dots ,e_{n+1}\}\subseteq H \subseteq \R^{n+1}$ where $H$, as in the whole of this section, denotes the hyperplane orthogonal to the vector $(1, \ldots, 1) \in \R^{n+1}$.
\begin{lemma}
  \label{non-symm-polytope-equivalence}
Let $K=\textup{conv}\{v_1, \ldots, v_{n+1}\} \subseteq \R^d$ ($n \ge d$) be a convex polytope with non-empty interior.
Let $T \colon \R^{n+1} \to \R^d$ be the linear map that $T e_i = v_i - v_0$ where $v_0 = \frac{1}{n+1}\sum_{i=1}^{n+1} v_i$
and $E = (\textup{ker}\,T)^\perp \subseteq \R^{n+1}$. Then $E \subseteq H$ is a subspace of dimension $d$ and
$K - v_0$ is linearly equivalent to $P_E \Delta_n$.
Consequently, $K$ is affinely equivalent to some orthogonal projection of the $n$-dimensional regular simplex $S_n$
onto a $d$-dimensional subspace.
\end{lemma}
\begin{proof}
  Clearly $(1, \ldots, 1) \in \textup{ker}\,T$, so $E \subseteq H$. Since $K$ has non-empty interior,
  vectors $v_i - v_0$ span the whole of $\R^d$, so $T$ is of full rank. Therefore $\textup{dim}\,E = d$
  and the argument given above applies.
\end{proof}

Now we can prove the following result concerning the isotropy constant of random projections of $B_1^n$ and $S_n$.
\begin{proposition}
  \label{prop-3-3}
There exist absolute constants $C, c_1, c_2 >0$ such that the Haar probability measure of the set of subspaces $E\in G_{n,d}$ verifying
$$
L_{P_E B_1^n}<C \qquad\textup{and}\qquad L_{P_E S_n}<C
$$
is greater than $1-c_1e^{-c_2\max\{\log n,d\}}$.
\end{proposition}
\begin{proof}

For small values of $d$, namely $d\leq c\log n$, the isotropy constant of a random projection is bounded by an absolute constant with probability greater than $1-\frac{c_1}{n^{c_2}}$ as a consequence of Dvoretzky's theorem. \bigskip

Let $G=(g_{ij})$ be a $d\times n$ Gaussian random matrix, i.e. the $g_{ij}$'s are i.i.d $\mathcal{N}(0,1)$
Gaussian random variables. Since $(\textup{ker}\,G)^\perp = \textup{Im}\,(G^t) \subseteq \R^n$, $G^t$ being the transpose matrix of $G$,
and the columns of $G^t$ are independent and rotationally invariant random vectors in $\R^n$, then a random
subspace $E = (\textup{ker}\,G)^\perp$ has dimension $d$ a.s. and is distributed
according to the Haar probability measure $\mu$ on $G_{n,d}$. Therefore for any constant $C > 0$,
$$
\mu\{E\in G_{n,d}\,;\; L_{P_EB_1^n}<C\} =\Pu\{L_{P_E B_1^n}<C\}.
$$
Lemma~\ref{symm-polytope-equivalence} and the affine invariance of the isotropy constant imply
$L_{P_E B_1^n} = L_{\textup{conv}\,(\pm G{e_1}, \ldots, \pm G{e_n})}$ a.s.
Klartag and Kozma proved in~\cite{KK} that if $C$ is a sufficiently large absolute constant,
\[
  \Pu\{L_{\textup{conv}\,(\pm G{e_1}, \ldots, \pm G{e_n})} < C \} > 1-c_1e^{-c_2d}
\]
which completes the proof in the symmetric case.\bigskip

For the non-symmetric case, we proceed analogously. For a $d \times (n+1)$ Gaussian random matrix $G = (g_{ij})$,
take $\bar{G} = (g_{ij} - \frac{1}{n+1}\sum_{k=1}^{n+1} g_{ik})_{i \le d, j \le n+1}$. Since the sum of the columns
of $\bar{G}$ is zero, $(\textup{ker}\,\bar{G})^\perp = \textup{Im}\,(\bar{G}^t) \subseteq H \subseteq \R^{n+1}$.
Moreover, since rows of $\bar{G}$ (equivalently, columns of $\bar{G}^t$) are independent canonical Gaussian random vectors in $H$,
the random subspace $E = (\textup{ker}\,\bar{G})^\perp \subseteq H$ is distributed according to the Haar probability measure on $G_{H, d}$
(Grassmann manifold of $d$-dimensional subspaces of $H$).
Lemma~\ref{non-symm-polytope-equivalence} and the affine invariance of the isotropy constant imply
$L_{P_E \Delta_n} = L_{\textup{conv}\,(G{e_1}, \ldots, G{e_{n+1}})}$ a.s.
Since $P_E \Delta_n = P_E (P_H \Delta_n)$ and $P_H \Delta_n$ is an $n$-dimensional regular simplex (in $H$),
a non-symmetric counterpart of the result of Klartag and Kozma~\cite{KK},
\[
  \Pu\{L_{\textup{conv}\,(G{e_1}, \ldots, G{e_{n+1}})} < C \} > 1-c_1e^{-c_2d},
\]
finishes the proof.
\end{proof}

In the final part of the section we will use the tools from Section 2 to prove the main result. In particular, whenever $d\geq cn$ the boundedness of the isotropy constant holds not only for ``most'' projections of $B_1^n$ and $S_n$ but deterministically for all of them.

\bigskip

\begin{theorem}
  \label{isotropic-constant-of-proj-of-b1n-ball}
  Let $E \subseteq \R^n$ be a subspace of dimension $1 \le d \le n-1$ and $\mathcal{K} = P_E B_1^n$,
  $\mathcal{T} = P_E S_n$. Then
  \[
    L_{\mathcal{K}}, L_{\mathcal{T}} \le C \sqrt{n/d}
  \]
  where $C > 0$ is a universal constant.
\end{theorem}

\begin{proof}
As an immediate consequence of (\ref{Milman-Pajor}),
\begin{equation}
  \label{isotropic-constant-main-ineq}
  L_{\mathcal{K}}^2 \le \frac{1}{d} \frac{1}{\Vol_d(\mathcal{K})^{2/d}}
                        \frac{1}{\Vol_d(\mathcal{K})} \int_{\mathcal{K}} |x|^2 \, \lambda_E(dx).
\end{equation}
Applying~(\ref{betke:main-inequality}), we obtain the bound
\begin{equation}
  \label{proj-of-b1n-ball:integral-estimate}
  \frac{1}{\Vol_d(\mathcal{K})} \int_{\mathcal{K}} |x|^2 \, \lambda_E(dx) \le
  \frac{1}{\Vol_d(\Delta_d)} \int_{\Delta_d} |x|^2 \, \lambda_{\textup{aff}\,\Delta_d}(dy) =
  \frac{2}{d+2}
\end{equation}
(for the last equality see e.g.~\cite[Lemma 2.3]{KK}).
 To estimate $\Vol_d(\mathcal{K})$ note that $n^{-1/2} B_2^n \subseteq B_1^n$, so $n^{-1/2} (B_2^n \cap E) \subseteq P_E B_1^n$.
 Therefore
  \begin{equation}
    \label{proj-of-b1n-ball:volume-estimate}
    \Vol_d(\mathcal{K})^{1/d} \ge \frac{c}{\sqrt{n d}}.
  \end{equation}
  Combining these two, we get
  \[
    L_{\mathcal{K}}^2 \le \frac{1}{d} \frac{n d}{c^2} \frac{2}{d+2} \le C' \frac{n}{d}.
  \]

  In the case of the simplex it is convenient to embed $E$ and $S_n$ into $H$. More precisely, we take
  $S_n = \textup{conv}\{P_H e_i \colon i = 1, \ldots, n+1\} \subseteq H \subseteq \R^{n+1}$
  and assume $E \subseteq H$. Now observe that $\mathcal{T} = P_E S_n = P_E \Delta_n$
  so~(\ref{betke:main-inequality}) again yields
  \[
    \frac{1}{\Vol_d(\mathcal{T})} \int_{\mathcal{T}} |x|^2 \, \lambda_E(dx) \le \frac{2}{d+2}.
  \]
  To bound the volume radius of $\mathcal{T}$ from below, we use
  the Rogers-Shephard inequality~\cite{rogers-shephard}:
  \[
    {2d \choose d}^{-1} \Vol_d(\mathcal{T} - \mathcal{T}) \le \Vol_d(\mathcal{T}).
  \]
  Note that
  \[ \begin{split}
    \mathcal{T} - \mathcal{T} &\supseteq \textup{conv}(\mathcal{T} \cup -\mathcal{T}) =
      \textup{conv}(P_E \Delta_n \cup -P_E \Delta_n) \\
      &= P_E \big(\textup{conv}(\Delta_n \cup -\Delta_n)\big)
       = P_E B_1^{n+1}.
  \end{split} \]
  Combining with the estimate~(\ref{proj-of-b1n-ball:volume-estimate}),
  \[ \begin{split}
    \Vol_d(\mathcal{T})^{1/d} &\ge {2d \choose d}^{-1/d} \Vol_d(P_E B_1^{n+1})^{1/d} \\
      &\ge {2d \choose d}^{-1/d} \frac{c}{\sqrt{(n+1)d}} \ge \frac{c'}{\sqrt{nd}}.
  \end{split} \]
\end{proof}
Due to Lemma~\ref{non-symm-polytope-equivalence}, we immediately get the following:
\begin{corollary}\label{isotropy-constant-polytopes}
  Let $K \subseteq \R^d$ be a non-degenerated ($\textup{dim}\, K = d$) convex polytope with $n$ vertices. Then
  \[
    L_K \le C \sqrt{\frac{n}{d}}.
  \]
\end{corollary}

\section{Isotropy constant of projections of random polytopes}

In this section we consider polytopes generated by the convex hull of vertices randomly chosen on the $S^{n-1}$. The main result is

\begin{proposition}
  \label{prop-4-1}
There exist absolute constants $C$, $c_1$ and $c_2$, such that if $m \ge n$, $\{P_i\}_{i=0}^m$ are independent random vectors on $S^{n-1}$ and $K=\textrm{conv}\{\pm P_1,\dots,\pm P_m\}$ or $K=\textrm{conv}\{P_0,\dots, P_m\}$, then
$$
\Pu\{L_{P_E K}\leq C\sqrt{\frac{n}{d}} \;\; \forall E\in G_{n,d}\;\forall \, 1 \le d \le n-1 \} \geq 1-c_1e^{-c_2n}.
$$

\end{proposition}

The proof follows \cite{A}.  We shall only sketch the main ideas as the technical computations can be found in that reference.

\begin{proof}[Sketch of the proof]
Let $E \subseteq \R^n$ denotes an $d$-dimensional subspace. The ideas in what follows will give us the proof for $m\geq cn$ with
an absolute constant $c$. If $m<cn$, Corollary~\ref{isotropy-constant-polytopes} gives deterministically $L_{P_E K}\leq C\sqrt{\frac{m}{d}}\leq
C'\sqrt{\frac{n}{d}}$.\medskip

Apply once again (1.1). Writing $r(K)$ the inradius of $K$ and using the inequality~(\ref{betke:main-inequality}), the main consequence of Proposition~\ref{betke:main-proposition}, we obtain that for any polytope $K\subset\R^n$ and any $d$-dimensional subspace $E$,
$$
L_{P_E K}^2\leq\frac{C}{r(K)^2} \max_{F \in \mathcal{F}_d(K)}
\frac{1}{\Vol_d(F)}\int_F |x|^2 \,\lambda_{\textup{aff} F}(dx).
$$

\medskip
When $K$ is the symmetric convex hull of $m$ independent random points in $S^{n-1}$, it was proved in \cite[Lemma 3.1]{A},
that for some constant $c$ such that $cn\leq m\leq ne^\frac{n}{2}$,
$$
\Pu\left\{r(K)<\frac{1}{2\sqrt{2}}\sqrt{\frac{\log\frac{m}{n}}{n}}\right\}\leq e^{-n}.
$$
The same proof gives the statement in the non-symmetric case.\medskip

On the other hand, with probability 1, each $d$-dimensional face
of $K$ is a simplex ${F}=\textrm{conv}\{Q_1,\dots,Q_{d+1}\}$ with
$Q_i=\varepsilon_i P_{j_i}$ (or just $Q_i = P_{j_i}$ in the non-symmetric case) where $1 \le j_1 < \cdots < j_{d+1} \le m$ and
$\varepsilon_i \in \{-1,1\}$. The same proof as in \cite{KK} and \cite{A} shows that with probability 1 we have
\begin{equation}
  \label{random-polytopes:identity-for-faces}
  \frac{1}{\Vol_d({F})}\int_F |x|^2 \,\lambda_{\textup{aff} F}(dx) = \frac{2}{d+2}+\frac{1}{(d+1)(d+2)}\sum_{i_1\neq i_2}^{d+1}\langle Q_{i_1},Q_{i_2}\rangle.
\end{equation}
%
In order to give a bound for this quantity for a fixed $F \in \mathcal{F}_d(K)$ we proceed in the same way as in \cite[Theorem 3.1]{A}, by using a version of Bernstein's inequality as stated in \cite{BLM}. We thus obtain
\begin{equation}
  \label{random-bodies:estimate-for-single-face}
\Pu\left\{\sum_{i_1\neq i_2}^{d+1}\langle Q_{i_1},Q_{i_2}\rangle>\epsilon(d+1)\right\}\leq2e^{-c\epsilon n}
\end{equation}
for every $\epsilon>\epsilon_0$, where $\epsilon_0$ is an absolute constant.\medskip

Now, for each $F \in \mathcal{F}_d(K)$ let $Q^F_1, \ldots, Q^F_{d+1}$ be vertices of $F$. Applying (\ref{random-bodies:estimate-for-single-face})
and the union bound over $\mathcal{F}_d(K)$ (whose cardinality is clearly bounded by ${2m \choose d+1}$), we obtain
for $\epsilon\log\frac{m}{n}>\epsilon_0$,
\begin{multline*}
\Pu\left\{\max_{F \in \mathcal{F}_d(K)}\sum_{i_1\neq i_2}^{d+1}\langle Q^F_{i_1},Q^F_{i_2} \rangle  >  \epsilon(d+1)\log\frac{m}{n}\right\} \\
  \leq   {2m \choose d+1} 2e^{-c\epsilon n\log\frac{m}{n}}
  \leq 2e^{-c\epsilon n\log\frac{m}{n}+(d+1)\log\frac{2em}{d+1}} \\
  \leq   2e^{-c\epsilon n\log\frac{m}{n}+n\log\frac{2em}{n}},
\end{multline*}
since the function $x\log\frac{C}{x}$ is increasing when $\frac{C}{x}>e$. Consequently, by the union bound over $d$,
\begin{multline*}
  \Pu\Big\{\exists \, 1\le d \le n-1 \ \textrm{s.t.} \
    \max_{F \in \mathcal{F}_d(K)}\sum_{i_1\neq i_2}^{d+1}\langle Q^F_{i_1},Q^F_{i_2} \rangle > \epsilon(d+1)\log\frac{m}{n}\Big\} \\
  \leq 2e^{-c\epsilon n\log\frac{m}{n}+n\log\frac{2em}{n}+\log n}
\end{multline*}

\medskip
Since $m\geq cn$, considering the complement set and using~(\ref{random-polytopes:identity-for-faces}), we can fix $\epsilon > 0$
a large enough numerical constant to obtain

\begin{multline*}
  \Pu\Big\{ \forall \, 1 \le d \le n-1 \ \max_{F \in \mathcal{F}_d(K)}\frac{1}{\Vol_d({F})}
   \int_F|x|^2 \,\lambda_{\textup{aff} F}(dx) \leq \frac{C}{d}\log\frac{m}{n} \Big\} \\ \ge
    1-2e^{-cn\log\frac{m}{n}}
\end{multline*}

\medskip
Thus, there exist constants $c, C > 0$ such that if $cn\leq m\leq ne^\frac{n}{2}$ then the set of points $(P_1,\dots,P_m)$ for which
the inequality $L_{P_E K}\leq C\sqrt\frac{n}{d}$ holds for every $d$-dimensional subspace $E$ and for every $1 \le d \le n-1$
has probability greater than $1-2e^{-cn\log\frac{m}{n}}-e^{-n}>1-c_1e^{-c_2n}$.

\medskip
In case $m>ne^\frac{n}{2}$, for $n$ large enough, $r(K)\geq\frac{1}{4}$ with probability greater than $1-e^{-n}$ so with this probability
$$
dL_{P_E K}^2\leq
\frac{1}{\Vol_d(P_E K)^\frac{2}{d}}\frac{1}{\Vol_d(P_E K)}\int_{P_E K}|x|^2dx\leq\frac{1}{\Vol_d(\frac{1}{4}B_2^d)^\frac{2}{d}}\leq
cd
$$
and the proof is complete.
\end{proof}

\section{A general result}

In this section we prove a general relation between the isotropy
constant of the hyperplane projections of an isotropic convex body
and of the body itself.
\medskip

\begin{corollary}\label{AlonsoMP} Let $K$ be an isotropic convex
body and let $H$ be a hyperplane. Then
$$
L_{P_H K}\sim L_K
$$
\end{corollary}

Its proof relies on the next Proposition which improves the
numerical constants appearing in a more general statement in
\cite{BKM} for the case of projections onto hyperplanes.

\begin{proposition}
Let $K\subset\R^n$ be an isotropic convex body and let
$H=\nu^\bot$ be a hyperplane. If $S(K)$ is the Steiner
symmetrization of $K$ with respect to $H$ then,
$$
\Big(1-c\ \frac{\log n}{n}\Big)L_K\le L_{S(K)}\le L_K
$$ for some numerical constant $c>0$.
\end{proposition}

\begin{proof}
Without loss of generality, we may assume that
$\nu=e_n=(0,\dots,0,1)$. Write $E=\langle e_n\rangle$ for the
1-dimensional subspace generated by $e_n$. The Steiner
symmetrization of $K$ is defined by
$$
S(K):=\left\{(y,t)\in\R^{n-1}\times\R\, ; \; y\in
P_H K,|t|\leq\frac{1}{2} \Vol_1(K\cap(x+E)) \right\}.
$$
Clearly $P_H K = P_H S(K) = S(K) \cap H$.

\bigskip
Now we study the inertia matrix of $S(K)$. First notice that for $x \in P_H K$,
$\Vol_1(K \cap (x + E)) = \Vol_1(S(K) \cap (x + E))$. For every $\theta \in S^{n-1} \cap H$, Fubini's theorem yields
\[ \begin{split}
  \int_{S(K)} \langle x, \theta \rangle^2 \,dx &= \int_{P_H K} \int_{S(K) \cap (x + E)} \langle y + t e_n, \theta \rangle^2 \,dt\,dy \\
    &= \int_{P_H K} \langle y, \theta \rangle^2 \Vol_1(K \cap (x + E)) \,dy = L_K^2.
\end{split} \]
Using the fact that $\int_{S(K) \cap (x + E)} t \, dt = 0$ for $x \in P_H K$, in the similar fashion we show that for every $\theta \in S^{n-1} \cap H$,
\[
  \int_{S(K)} \langle x, \theta \rangle \langle x, e_n \rangle \,dx = 0.
\]
Also $\int_{S(K) \cap (x + E)} t^2 \, dt \le \int_{K \cap (x + E)} t^2 \, dt$ for $x \in P_H K$, thus
\[
  \int_{S(K)} \langle x, e_n \rangle^2 \,dx \le \int_{K} \langle x, e_n \rangle^2 \,dx = L_K^2.
\]
Taking $\sigma > 0$ such that
$\sigma^2 := \int_{S(K)} \langle x, e_n \rangle^2 \,dx / L_K^2$, we obtain that the inertia matrix of $S(K)$ is
\[
  M = L_K^2 \begin{pmatrix} 1& & & \\ &\ddots& & \\ & & 1&\\ & & & \sigma^2 \end{pmatrix}.
\]
The volume of $S(K)$ is 1, so $L_{S(K)} = (\textup{det} M)^{1/2n}$ (see~\cite{MP}), which means
\begin{equation}\label{steiner}
  L_{S(K)} = \sigma^{1/n} L_K = \left(\frac{\left(\int_{S(K)}\langle
x,e_n\rangle^2\,dx\right)^{1/2}}{L_K}\right)^{1/n} L_K.
\end{equation}
Since $\sigma \le 1$, we obtain $L_{S(K)} \le L_K$.

\bigskip
A well-known fact due to Hensley~\cite{H} states that
$\Vol_{n-1}(K_1 \cap H) \sim \left( \int_{K_1}\langle x,e_n\rangle^2 \,dx \right)^{-1/2}$
for any convex body $K_1$ with volume 1 and center of mass at the origin. Using this fact for $K$ and $S(K)$ in~(\ref{steiner})
we obtain that for some absolute constant $c > 0$,
$$
L_{S(K)}\geq\left(c\frac{\Vol_{n-1}(K\cap H)}{\Vol_{n-1}(S(K)\cap H)}\right)^{1/n} L_K
= \left(c\frac{\Vol_{n-1}(K\cap H)}{\Vol_{n-1}(P_H K)}\right)^{1/n} L_K.
$$
Now we use the following inequalities:
\[ \begin{split}
\Vol_1(P_E K) \Vol_{n-1}(K\cap H) &\geq c_1 \Vol_n(K), \\[1ex]
\frac{1}{n}\Vol_{n-1}(P_H K) \Vol_1(K\cap E) &\leq \Vol_n(K) = 1.
\end{split} \]
(For the proof, see for instance~\cite[Lemma 8.8]{P} where the first inequality is proved for a symmetric body $K$ with $c_1 = 1$ and the non-symmetric case can be proved similarly with an absolute constant $c_1 > 0$. The proof of the second inequality given in~\cite{P} works in the non-symmetric case.)
They yield
$$
\Vol_{n-1}(K\cap H) \geq \frac{c_1}{\Vol_1(P_E K)}\geq\frac{c_1}{2R(K)}
$$
and
$$
\Vol_{n-1}(P_H K)\leq\frac{n}{\Vol_1(K \cap E)}\leq\frac{n}{2r(K)}
$$
where $R(K)$ and $r(K)$ are the circumradius and the inradius of $K$ respectively.

\bigskip
Since every isotropic convex body verifies $R(K)\leq(n+1)L_K$ and $r(K)\geq L_K$
(see~\cite{G} or~\cite{KLS}, for instance) we obtain
\[
  L_{S(K)} \ge \left(\frac{c c_1}{n(n+1)}\right)^{1/n} L_K.
\]
\end{proof}

\noindent\emph{Proof of Corollary \ref{AlonsoMP}}. Since
$P_H K = S(K)\cap H$  we have
$$
L_{P_H K}=L_{S(K)\cap H}\sim L_{S(K)}\sim L_K
$$
where the first equivalence is the corresponding one for sections of convex bodies
as proved in~\cite{MP}.\hfill$\Box$

\subsection*{Acknowledgements}

The fourth named author would like to thank Prof. Stanis\l{}aw Szarek for some useful discussions.

Part of this work has been done while the fourth named author enjoyed the hospitality of University of Zaragoza
staying there as Experienced Researcher of the European Network PHD.


\end{document}